\documentclass[10pt]{amsart}
\usepackage{amsmath,amscd}
\usepackage{amsbsy}
\usepackage{amssymb}
\usepackage{amscd,amsthm}
\usepackage[all,cmtip]{xy}
\usepackage[russian,english]{babel}
\usepackage[utf8]{inputenc}
\usepackage{xcolor}
\usepackage{dsfont}

\usepackage{hyperref}

\renewcommand{\epsilon}{\varepsilon}

\renewcommand{\ell}{x}

\newtheorem{thm}{Theorem}\numberwithin{thm}{section}

\begin{document}
	\begin{center}
		\huge{A comment on the number of $k$-th powers inside arithmetic progressions}\\[1cm] 
	\end{center}
	\begin{center}
		\large{Sa$\mathrm{\check{s}}$a Novakovi$\mathrm{\acute{c}}$}\\[0,5cm]
		{\small July 2026}\\[0,2cm]
	\end{center}
		{\small \textbf{Abstract}. 
			In \cite{BD} Bourgain and Demeter found sharp upper bounds for the number of $k$-th powers inside arbitrary arithmetic progressions whose step has $O(1)$ many divisors. We make the easy observation that the same arguments are still valid if the step does not grow too rapidly in relation to the length of the progression. Furthermore, we give sharp bounds for the number of $k$-th powers among the first $N$ terms for $N$ large enough. Both results should be known. Nevertheless, we add to the literature.}
		\section{Introduction}
		\noindent
For fixed $k\ge 2$, how many $k$-th powers of integers can lie inside an arithmetic progression of length $N$? By considering the progression $\{1,2,\ldots, N\}$ we see that this number can be as large as $\sim N^{\frac1k}$. More precisely, let $\mathcal{Q}_k(N;q,a)$ denote the number of $k$-th powers  in the arithmetic progression $a+q,a+2q,\ldots,a+Nq$. Write
$$\mathcal{Q}_k(N)=\sup_{a,q\in\mathbb{N}\atop{q\not=0}}\mathcal{Q}_k(N;q,a).$$
Rudin \cite{R} has conjectured that $\mathcal{Q}_2(N)\sim N^{\frac12}$ and Bombieri, Granville and Pintz \cite{BGP} that 
$$
\mathcal{Q}_k(N)\lesssim N^{(\frac12+\epsilon)},
$$
for $k\geq 4$. Bourgain and Demeter \cite{BD} similarly conjectured that for each $k\ge 2$
\begin{equation}
	\label{ucuycfiedoie}
	\mathcal{Q}_k(N)\lessapprox N^{\frac1k},
\end{equation}
where $\lessapprox$ denotes logarithmic losses of the form $(\log N)^{O(1)}$. The logarithmic loss here is added for extra safety and up to now, it is not clear whether it is really needed.

The best known upper bound for $\mathcal{Q}_2(N)$ is due to Bombieri and Zannier \cite{BZ}
$$\mathcal{Q}_2(N)\lessapprox N^{\frac35}.$$
This builds on earlier work \cite {BGP} of Bombieri, Granville and Pintz that proved the result with exponent $\frac23$ in place of $\frac35$. These rely on deep results in number theory regarding rational points on  curves. 
The papers \cite{G} and \cite{CG} contain a nice discussion on the problem.
In the special case when the step $q$ of the progression has $O(1)$ many divisors, Bourgain and Demeter proved the following:
\medskip

\begin{thm}[\cite{BD}, Theorem 0.1]
		Let $d(q)$ be the number of divisors of $q$. Then for each polynomial $P_k$ of degree $k\ge 1$ with integer coefficients and each $a,q,N\in\mathbb{Z}$ we have   
	$$
	|\{t\in \mathbb{Z}:\;P_k(t)\in\{a+q,a+2q,\ldots,a+Nq\}\}|\lesssim d(q)^{k-1}N^{\frac1k}.
	$$
	The implicit constant depends only on $k$.
\end{thm} 
The proof uses a nice induction argument. In this short remark, we observe that the same induction argument can be used to prove a similar statement in the case where the step $q$ satisfies $d(q)\lesssim N^{\frac 1{k(k+1)}}\cdot \mathrm{log}(N)^{O(1)}$. In fact, we observe:
\begin{thm}
	Let $d(q)$ be the number of divisors of $q$ and assume that $d(q)\lesssim N^{\frac 1{k(k+1)}}\cdot \mathrm{log}(N)^{O(1)}$ . Then for each polynomial $P_k$ of degree $k\ge 1$ with integer coefficients and each $a,q,N\in\mathbb{Z}$ we have   
	$$
	|\{t\in \mathbb{Z}:\;P_k(t)\in\{a+q,a+2q,\ldots,a+Nq\}\}|\lessapprox N^{\frac2k}.
	$$
	The implicit constant depends only on $k$.
\end{thm} 
\begin{proof}
	We use induction on $k$. The case $k=1$ is obvious. We may assume the statement holds for $k>1$. Let $P_{k+1}$ be a polynomial of degree $k+1$. Fix $t_0\neq t$ such that
		$$
		P_{k+1}(t),P_{k+1}(t_0)\in\{a+q,a+2q,\ldots,a+Nq\}.
		$$
	Write $$P_{k+1}(t)-P_{k+1}(t_0)=(t-t_0)P_k(t),$$
	for some polynomial $P_k$ of degree $k$.
	Notice that we must have
	\begin{equation}
		\label{wqwdiouc7rcyf9ex}
		\begin{cases}(t-t_0)=n_1q_1\\P_k(t)=q_2n_2\end{cases}
	\end{equation}
	with  $q_1q_2=q$. Moreover $n_1n_2\le N$. This yields that either $$n_1\le N^{\frac1{k+1}}$$
	or 
	$$n_2\le N^{\frac{k}{k+1}}.$$
	Let us fix the pair $(q_1,q_2)$. In the first case, there are $\le N^{\frac1{k+1}}$ possible values of $t$ (considering only the first equation in \eqref{wqwdiouc7rcyf9ex}), while in the second case there are $O(d(q_2)^{k-1}N^{\frac1{k+1}})$ values of $t$, due to the induction hypothesis (considering only the second equation in \eqref{wqwdiouc7rcyf9ex}). Since there are $d(q)$ many ways to choose the pair $(q_1,q_2)$ we conclude that the total contribution is 
	$$\lesssim d(q)(N^{\frac1{k+1}}+d(q)^{k-1}N^{\frac1{k+1}})\lesssim d(q)^kN^{\frac1{k+1}}.$$
	By assumption, we have
	$$
	d(q)\lesssim N^{\frac 1{k(k+1)}}\cdot \mathrm{log}(N)^{O(1)}.
	$$
	But this shows that the total contribution is actually
	$$
	\lesssim d(q)(N^{\frac1{k+1}}+d(q)^{k-1}N^{\frac1{k+1}})\lesssim d(q)^kN^{\frac1{k+1}}\lesssim N^{\frac2{k+1}}\cdot \mathrm{log}(N)^{O(1)}.
	$$
\end{proof}
\noindent
In \cite{HP} Hajdu and Papp gave bounds for the number of $k$-th powers and arbitrary powers among the first $N$ terms of an arithmetic progression, for $N$ large enough. Write $P_{a,q;N}(k)$ for the number of $k$-th powers among the first $N$ terms of the arithmetic progression $a+qx\quad (x\geq 0)$. Among others, they proved the following result.
\begin{thm}[\cite{HP}, Theorem 2.1]
	For every $\epsilon >0$ there is an $k_0$ depending on $\epsilon$ such that for any $k>k_0$ we have $P_{a,q;N}(k)\leq (1+\epsilon)N^{\frac 1k}$, whenever $N>N_0$. Here $N_0=N_0(\epsilon,k,a,q)$. 
\end{thm}
\noindent
In a similar spirit, we make the following observation.
\begin{thm}
		Let $r>0$ be a positive integer and assume $q\lesssim N^r$. For each $\epsilon >0$ and each polynomial $P_k$ of degree $k\ge 1$ with integer coefficients there exists a $N_0$ depending on $\epsilon$ and $k$ such that for each $a,q,N\in\mathbb{Z}$ we have   
	$$
	|\{t\in \mathbb{Z}:\;P_k(t)\in\{a+q,a+2q,\ldots,a+Nq\}\}|\lesssim N^{(\frac 1k + \epsilon)}.
	$$
	whenever $N> N_0$. The implicit constant depends only on $k$.
\end{thm}
\begin{proof}
		Again we use induction on $k$. The case $k=1$ is obvious. The rest of the proof is very similar to the proof of Theorem 1.2. Nevertheless, we all the details. So, we may assume the statement holds for $k>1$. Let $P_{k+1}$ be a polynomial of degree $k+1$. Fix $t_0\neq t$ such that
	$$
	P_{k+1}(t),P_{k+1}(t_0)\in\{a+q,a+2q,\ldots,a+Nq\}.
	$$
	Write $$P_{k+1}(t)-P_{k+1}(t_0)=(t-t_0)P_k(t),$$
	for some polynomial $P_k$ of degree $k$.
	Notice that we must have
	\begin{equation}
		\label{wqwdiouc7rcyf9ex}
		\begin{cases}(t-t_0)=n_1q_1\\P_k(t)=q_2n_2\end{cases}
	\end{equation}
	with  $q_1q_2=q$. Moreover $n_1n_2\le N$. This yields that either $$n_1\le N^{\frac1{k+1}}$$
	or 
	$$n_2\le N^{\frac{k}{k+1}}.$$
	Let us fix the pair $(q_1,q_2)$. In the first case, there are $\le N^{\frac1{k+1}}$ possible values of $t$ (considering only the first equation in \eqref{wqwdiouc7rcyf9ex}), while in the second case there are $O(d(q_2)^{k-1}N^{\frac1{k+1}})$ values of $t$, due to the induction hypothesis (considering only the second equation in \eqref{wqwdiouc7rcyf9ex}). Since there are $d(q)$ many ways to choose the pair $(q_1,q_2)$ we conclude that the total contribution is 
	$$\lesssim d(q)(N^{\frac1{k+1}}+d(q)^{k-1}N^{\frac1{k+1}})\lesssim d(q)^kN^{\frac1{k+1}}.$$
By a classical result of Wigert \cite{W}, we have 
	$$
	\underset{q\leq cN^r}{\mathrm{max}}\ d(q)\leq (cN^r)^{\frac{\mathrm{log}(2)+O(1)}{\mathrm{log}\mathrm{log}(cN^r)}}.
	$$
	But this shows that 
	$$
	d(q)^k\leq (cN^r)^{\frac{(\mathrm{log}(2)+O(1))k}{\mathrm{log}\mathrm{log}(cN^r)}}.
	$$
	By taking the logarithm on both sides, it is easy to see that for a given $\epsilon >0$, there is a $N_0$ depending on $\epsilon$ and $k$ such that 
	$$
	(cN^r)^{\frac{(\mathrm{log}(2)+O(1))k}{\mathrm{log}\mathrm{log}(cN^r)}}\leq N^{\epsilon}
	$$
	whenever $N>N_0$. This shows that the total contribution is
	$$
	\lesssim d(q)(N^{\frac1{k+1}}+d(q)^{k-1}N^{\frac1{k+1}})\lesssim d(q)^kN^{\frac1{k+1}}\lesssim N^{(\frac 1{k+1}+ \epsilon)}.
	$$
\end{proof}

		\vspace{0.3cm}
		\noindent
		{\tiny HOCHSCHULE FRESENIUS UNIVERSITY OF APPLIED SCIENCES 40476 D\"USSELDORF, GERMANY.}\\
		E-mail adress: sasa.novakovic@hs-fresenius.de\\
		
	\end{document}